\documentclass[letterpaper,10pt,nomarginpar,leqno,twoside]{amsart}

\usepackage{setspace}
\usepackage{lastpage}

\usepackage{ccaption}
\usepackage{ifpdf}
\usepackage{geometry}

\usepackage[T1]{fontenc}
\usepackage{lmodern} % This is for the Sans Serif; must come before math design
\usepackage{charter} 
\usepackage{amsfonts,amssymb}
\usepackage{eucal}
\usepackage[mathcal,mathbf,text-hat-accent]{euler}% somehow, this 
\usepackage[pdftex,dvipsnames,usenames]{color}%
\usepackage{amsmath,amstext,amsopn,amsxtra}

\usepackage{braket}
\usepackage{tensor}
\usepackage{slashed}

\usepackage{enumitem}
\setenumerate{label=(\arabic*)}

\usepackage{textcomp} %%%%% currency
\usepackage{bbm}%%%% can write blackboard for lowercase alphabets

\usepackage{pifont}%
\usepackage{cancel}
\usepackage{feynmp}

\usepackage[all]{xy}

\usepackage{yhmath} % extra wide deliminators %interferes with eulervm, either one will kill the other.

\usepackage[lite,initials]{amsrefs} %should come after hyperref

\usepackage[pdftex]{hyperref}
\usepackage[pdftex]{graphicx}
\usepackage{epstopdf}
\DeclareGraphicsExtensions{.pdf,.png,.jpg,.eps}

\hypersetup{naturalnames=false,%
	bookmarks=true,bookmarksopenlevel=2,%
	colorlinks=true,%	% links are colored
	linkcolor=shblue,%         % color of normal internal links
	citecolor=shblue,%	  % color of bibliographical cite links
	pagecolor=shblue,%         % color of links to other page	
	menucolor=red,%         % color of Acrobat Reader menu buttons
	urlcolor=shblue,%       % color of page of \url{...}
	urlbordercolor={0 1 0},%
	}

\newcommand{\defaultthmheaderfont}{}

\captiondelim{.}

\numberwithin{equation}{section}

\usepackage{amsthm}
\usepackage{framed}
	\providecommand{\thref}[1]{({\defaultthmheaderfont\ref{#1}})}

\definecolor{shadecolor}{RGB}{255,255,214}
\newtheoremstyle{numberonly}% ⟨name⟩ 
	{2ex}%	⟨Space above⟩ 
	{1ex}%	⟨Space below⟩ 
	{}%	⟨Body font⟩
	{}%	⟨Indent amount⟩
	{\defaultthmheaderfont}% ⟨Theorem head font⟩ 
	{}%	⟨Punctuation after theorem head⟩ 
	{.5em}%	⟨Space after theorem head⟩2 
	{(\thmnumber{#2})\thmnote{ \textbf{#3}}}%	⟨Theorem head spec (can be left empty, meaning ‘normal’)⟩ 
\newtheoremstyle{theorem}% 
	{2ex}%	⟨Space above⟩ 
	{1ex}%	⟨Space below⟩ 
	{\itshape}%	⟨Body font⟩
	{}%	⟨Indent amount⟩
	{\defaultthmheaderfont}% ⟨Theorem head font⟩ 
	{}%	⟨Punctuation after theorem head⟩ 
	{.5em}%	⟨Space after theorem head⟩2 
	{(\thmnumber{#2})\thmname{ \textbf{#1}}\thmnote{ \textbf{(#3)}}}%	⟨Theorem head spec (can be left empty, meaning ‘normal’)⟩ 
\newtheoremstyle{remark}% ⟨name⟩ 
	{2ex}%	⟨Space above⟩ 
	{1ex}%	⟨Space below⟩ 
	{}%	⟨Body font⟩
	{}%	⟨Indent amount⟩
	{\itshape}% ⟨Theorem head font⟩ 
	{.}%	⟨Punctuation after theorem head⟩ 
	{.5em}%	⟨Space after theorem head⟩2 
	{\thmname{#1}\thmnote{ (#3)}}%	⟨Theorem head spec (can be left empty, meaning ‘normal’)⟩ 

\theoremstyle{numberonly}
\newtheorem{para}[equation]{Subsection}

\theoremstyle{theorem}%pmargin
\newtheorem{thm}[equation]{Theorem}         
      
\newtheorem{cor}[equation]{Corollary}         
\newtheorem{lem}[equation]{Lemma}

\theoremstyle{plain}

\theoremstyle{remark}

\definecolor{shbluegray}{RGB}{170,187,204} %NONAME {rgb}{.582,.672,.809} % 149 172 207 {cmyk}{.33,.11,.01,0}
\definecolor{shblue}{RGB}{39,64,139} %RoyalBlue {rgb}{.227,.367,.605} % 58 94 155 {cmyk}{.74,.46,0,0}
\definecolor{shorange}{RGB}{218,165,32} %GoldenRod {rgb}{.762,.566,.129} % 195 145 33 {cmyk}{0.04,.27,.98,.13}
\definecolor{shgreen}{RGB}{154,205,50} %YellowGreen %{rgb}{.453,.680,0} % 116 174 0 {cmyk}{0,0,0.88,0.15}
\definecolor{shyellow}{RGB}{255,255,214} %LightYellow 245 243 202 {cmyk}{0,0,.25,0}  252 252 198

\newcommand{\de}[1]{\textit{#1}} % defined word typeface

\newcommand{\Cas}{\mathrm{Cas}}

\def\ff.{if and only if}
\def\oto.{one-to-one}
\def\otoc.{one-to-one correspondence}

\DeclareMathOperator{\vol}{vol} %%%%%%%%%%%% volume form

\newcommand{\eq}[1]{ %%%%%%%%%shorthand command for equations
\begin{equation}%
#1 %
\end{equation}%
}
\newcommand{\dm}[1]{%  %%%%%%%%%shorthand command for displaymath
\begin{displaymath}%
#1 %
\end{displaymath}%
}

\newcommand{\ga}[1]{  %
\begin{gather*}%
#1 %
\end{gather*}%
}
\newcommand{\al}[1]{  %
\begin{align*}%
#1
\end{align*}%
}

\def\rhs.{right-hand side}
\newcommand{\mf}[1]{\mathfrak{#1}}
\newcommand{\Lap}{\triangle}

\newcommand{\cgr}{\gamma}
\newcommand{\dgr}{\delta}

\newcommand{\lgr}{\lambda}

\newcommand{\rgn}{\Omega}%\mathscr{D}}
\newcommand{\prt}{\partial}
\newcommand{\R}{{\mathbb{R}}} %%%%%%%%%%%% real set
\newcommand{\rar}{\rightarrow}
\newcommand{\Cgr}{\Gamma}
\newcommand{\grg}{\mf{g}}
\newcommand{\grt}{\mf{t}}

\def\SL #1.{\mathrm{SL}(#1)} %%%%%%%%% special linear
\def\sl #1.{\mf{sl}({#1})} %%%%%%%%% special linear
\def\piff #1 #2.{\frac{\partial #1}{\partial #2}}
\def\vpiff #1 #2.{\frac{\partial }{\partial #2}#1}
\def\SU #1.{\mathrm{SU}({#1})} %%%%% special unitary group
\def\su #1.{\mf{su}({#1})} %%%%% unitary group
\def\U #1.{\mathrm{U}({#1})} %%%%% unitary group

\newcommand{\inv}[1][1]{^{-{#1}}} %%%%%%%%%%%%%%%%%% inverse
\DeclareMathOperator{\Exp}{Exp}%%%%%%% exponential map
\DeclareMathOperator{\Ad}{Ad}%%%%%%%%%%%% adjoint representation
\DeclareMathOperator{\ad}{ad}%%%%%%%%%
\DeclareMathOperator{\End}{End} %%%% Endomorphisms
\newcommand{\ldef}{:=}

\newcommand{\brk}[1]{\braket{ #1 }}
\newcommand{\ecbrk}{\brk{ \ , \ }}
\newcommand{\kgr}{\kappa}
\DeclareMathOperator{\tr}{tr}%%%%%%%%% trace
\providecommand{\comp}{\circ} %%%%%%%%%%%%%% fuction composition

\newcommand{\cinfty}{\ensuremath{C^\infty}} %%%%%%%%%%%%% c infinity, smooth
\renewcommand{\div}{\divergence}
\DeclareMathOperator{\grad}{grad}%%%%%%%%%%%%%%%%%%%% gradient
\DeclareMathOperator{\divergence}{div}%%%%%%%%%%%%%%%%%%%% div
\newcommand{\mfX}{{\mf{X}}}
\newcommand{\lie}{L}%{\mathpzc{L}}
\newcommand{\Id}{\mathbf{1}} %%%% the identity operator
\newcommand\spmat[1]{\left(\begin{smallmatrix} #1
                           \end{smallmatrix}\right)} %%% small matrix
\newcommand\spmtx[1]{\spmat{#1}} %%% small matrix

\def\ediff #1 #2 #3.{\left.\frac{d#1}{d#2}\right|_{#3}}
\def\veviff #1 #2 #3.{\left.\frac{d}{d#2}#1\right|_{#3}}
\def\vediff #1 #2 #3.{\left.\frac{d}{d#2}#1\right|_{#3}}
\newcommand{\uen}{\mcU} % universal enveloping algebra
\newcommand{\mc}[1]{\mathcal{#1}}
\newcommand{\mcU}{{\mc{U}}}
\newcommand{\mcZ}{{\mc{Z}}}
\newcommand{\tsr}{\otimes} %%%%%%%%%%%%%%% tensor product
\newcommand{\wt}[1]{\widetilde{#1}}
\newcommand{\cen}{\mcZ} % universal enveloping algebra
\def\nbhd.{neighborhood}
\def\wrt.{with respect to}
\newcommand{\del}{\nabla}
\newcommand{\hlf}{{\frac{1}{2}}} %%%%%%%%%%%%%%%% 1/2 
\let\ge=\geqslant %%%%%%%%%%%%%%%%% greater than or equal to
\let\le=\leqslant %%%%%%%%%%%%%%%%% less than or equal to
\DeclareMathOperator{\Spec}{Sp} % spectrum

\newcommand{\spec}{\Spec}
\DeclareMathOperator{\Duf}{Duf}
\def\onb.{orthonormal basis}
\def\poincare.{Poincar\'{e}}

\begin{document}
\title[Asymptotic expansion of the heat kernel on a compact Lie group]{The asymptotic expansion of the heat kernel on a compact Lie group}
\author{Seunghun Hong}
\address{Department of Mathematics\\The Pennsylvania State University\\University Park, PA 16802 U.S.A.}
\email{hong@math.psu.edu}
\urladdr{http://www.math.psu.edu/hong}
\thanks{The author wishes to express his heartfelt gratitude to his advisor, Professor N. Higson, for the kind guidance and advice.}

\begin{abstract}
Let $G$ be a compact connected Lie group equipped with a bi-invariant metric. We calculate the asymptotic expansion of the heat kernel of the laplacian on $G$ and the heat trace using Lie algebra methods. The Duflo isomorphism plays a key role.
\end{abstract}

\maketitle

\tableofcontents

\section{Introduction}

\begin{para}
The laplacian $\Lap$ on a compact riemannian manifold $M$ depends only on the metric on $M$. Conversely, one can deduce the metric from the laplacian. Hence, it is reasonable to expect the spectrum $\spec(\Lap)$ of the laplacian to be constrained by the geometry and vice versa. The first result in this vein is Weyl's law \cites{weyllap1,weyllap2} which states that the number $N(\lgr)$ of the eigenvalues of the laplacian less than $\lgr$ satisfies the asymptotic equality
\dm{ \frac{N(\lgr)}{\lgr^{n/2}} = \frac{\vol(\rgn)}{(4\pi)^{n/2}\Cgr(\frac{n}{2}+1) } +O(1/\lgr)}
as $\lgr\rar\infty$, where $\rgn$ is a bounded open subset of $\R^2$ or $\R^3$ on which the laplacian is defined. G{\r{a}}rding \cite{garding} proved the higher-dimensional case (for generic elliptic operators). On closed riemannian manifolds, the same law was proved for the laplacian by Duistermaat and Guillemin \cite{duistermaatguillemin}, and for generic elliptic operators by Minakshisundaram and Pleijel \cite{minak}. 

Weyl's law can be reformulated as an asymptotic behavior of the function
\dm{ Z(t) = \sum_{k=1}^\infty e^{-t\lgr_k}}
where $-\lgr_k$ denotes the eigenvalue of the $k$th eigenfunction of the laplacian. This funcion resembles the ``partition function'' in physics---a function that is often invariant under the symmetry of the physical system it describes. It is the trace of the heat diffusion operator
\dm{ e^{t\Lap} = \spmtx{ e^{-t\lgr_1}\\
	&e^{-t\lgr_2}\\
	&&e^{-t\lgr_3}\\
	&&&\ddots}.}
The relation between $Z(t)$ and the number of eigenfunctions becomes evident as we consider the limit $t\rar0+$; in that limit, the partial sum 
\dm{\sum_{k=1}^K e^{-t\lgr_k}}
converges to $K$, the number of the eigenvalues from $\lgr_1$ to $\lgr_K$. Weyl's law can be shown to be equivalent to the asymptotic law
\eq{  t^{n/2}Z(t)  \sim \frac{\vol(\rgn)}{(4\pi)^{n/2}} +O(t)\label{eq:weylthmlap}}
as $t\rar0+$. Minakshisundaram and Pleijel \cite{minak} showed that the partition function has an asymptotic expansion
\dm{ Z(t)\sim \Bigl(\frac{1}{4\pi t}\Bigr)^{\dim M/2} (a_0+a_1t+a_2t^2 +\dotsb)}
as $t\rar0+$. McKean and Singer \cite{mckeansinger} calculated $a_0,a_1,a_2$ and, in particular, showed that $a_0$ is the riemannian volume of $M$ and $a_1$ is $\frac{1}{6}\int_M S$ where $S$ is the scalar curvature of $M$. The higher order coefficients are extremely hard to calculate in general.

Our goal is to consider the asymptotic expansion of $Z(t)$ in the case where $M$ is a compact connected Lie group $G$, equipped with a metric that is invariant under the left and right-translations. We follow the heat kernel method, that is, we calculate the asymptotic expansion of the heat kernel $k_t(x,y)$ of the laplacian. The asymptotic expansion for $Z(t)$ can then be obtained from the relation
\dm{ Z(t) = \int_G k_t(x,x)\,\vol(x).  }
Here $\vol(x)$ is the riemannian volume form. Our motivation comes from the expectation that the high degree of symmetry will substantially simplify the calculations. 

Our strategy is to utilize the tight connection between $G$ and its Lie algebra $\grg$. The key ingredient is the Duflo isomorphism
\dm{ \Duf:S(\grg)^\grg\rar \cen(\grg). }
The space $S(\grg)^\grg$ can be identified with the constant coefficient differential operators on $\grg$, and $\cen(\grg)$ can be identified with the bi-invariant differential operators on $G$. Owing to the bi-invariance of the metric on $G$, the laplacian $\Lap_G$ on $G$ is a bi-invariant differential operator; and it corresponds to the Casimir element in $\cen(\grg)$. The operator on $\grg$ corresponding to the preimage of the Casimir under the Duflo isomorphism is (not surprisingly) the laplacian $\Lap_\grg$ on the euclidean space $\grg$ with an extra constant term. The heat kernel of the flat laplacian is simply the gaussian kernel. Based on these relations, we can deduce, without much effort, the asymptotic expansion of the heat kernel of the laplacian on $G$.
\end{para}

\begin{para}[Notations]
Throughout this article $G$ denotes a compact connected Lie group, and $\grg$ its Lie algebra, namely, the tangent space $T_eG$ at the identity $e\in G$. We denote by $\wt X$ the left-invariant vector field on $G$ generated by $X\in\grg$. We denote by $\ecbrk$ the (selected) bi-invariant metric on $G$. Such a metric is equivalent to an $\Ad(G)$-invariant inner product on $\grg$.
\end{para}

\section{Preliminaries}

\begin{para}
We review some basic analytic and algebraic notions related to the laplacian on a compact  Lie group.
\end{para}

\subsection{Analytic Aspects}

\begin{para}
Proofs for most of the statements made in this subsection can be found in \cite{bgv}*{Ch.2}.
\end{para}

\begin{para}[The Laplacian]
Let $\cinfty(G)$ denote the space of smooth functions on $G$, and let $\mfX(G)$ denote the space of smooth vector fields on $G$. Define the \de{gradient} operator $\grad:\cinfty(G)\rar \mfX(G)$ and the \de{divergence} operator $\div:\mfX(G)\rar\cinfty(G)$ by 
\ga{ \brk{\grad f,X}_\kgr=Xf,\\
  (\div X)\vol = \lie_X\vol, }
for $X\in\mfX(G)$, where  $\vol$ is the riemannian volume form and $\lie_X$ is the Lie derivative with respect to $X$. Then the \de{laplacian} (or the \de{Laplace-Beltrami operator}) $\Lap_G:\cinfty(G)\rar\cinfty(G)$ is defined by
\dm{ \Lap_G f = \div(\grad f).}
The definition of $\Lap_G$ is independent of local coordinates and depends only on the metric. Because our metric $\ecbrk$ is bi-invariant, so is the laplacian. 

An expression for $\Lap_G$ in local coordinates $(x_1,\dotsc,x_n)$ can be given as follows. Let $g$ be the matrix defined by $g_{ij}=\brk{\prt_i,\prt_j}$ where $\prt_i=\prt/\prt x_i$. Let $g^{ij}$ denote the $(i,j)$-entry of $g\inv$. Then,
\eq{ \Lap_G f = \frac{1}{\sqrt{\det g}} \sum_{i,j}\prt_i(\sqrt{\det g} g^{ij}\prt_j f).\label{eq:lapbeltlocex}}
\end{para}

\begin{para}[The Spectrum of the Laplacian]
So far the laplacian is an unbounded operator whose domain $\cinfty(G)$ is a dense subspace of $L^2(G)$. The domain can be extended to the Sobolev space $H^2(G)$, that is, the space of measurable functions $u$ on $G$ such that the norm $\|u\|_{H^2}=\|u\|_{L^2}+\sum_{i}\|\prt_{i} u\|_{L^2}+\sum_{i,j}\|\prt_i\prt_j u\|_{L^2}$ is finite. The Sobolev embedding theorem tells us that $H^2(G)$ is a subspace of $L^2(G)$ and that the inclusion map is compact. The extension
\dm{ \Lap_G: H^2(G)\rar L^2(G) }
is the unique self-adjoint extension of $\Lap_G$ which was originally defined on $\cinfty(G)$. In the language of the theory of unbounded operators,  the laplacian is \de{essentially self-adjoint} on $\cinfty(G)$.

It turns out that $(\Id-\Lap_G)$, where $\Id$ denotes the identity operator, admits an inverse that is compact. Then the spectral theorem implies that the eigenfunctions $\set{u_k}_{i=1}^\infty$ of $(\Id-\Lap_G)\inv$ form an orthonormal basis for $L^2(G)$. Owing to the regularity of elliptic differential operators, the eigenfunctions are of $\cinfty$. We can also conclude from the spectral theorem that the eigenfunctions $u_k$ can be ordered in such a way that the corresponding eigenvalues $-\lgr_k$ of $\Lap_G$ give a nonincreasing unbounded sequence of negative real numbers,
\dm{ 0> -\lgr_1\ge -\lgr_2\ge-\lgr_3\ge\dotsb.}
The \de{heat diffusion operator} of $\Lap_G$ is then defined by the matrix
\dm{ e^{t\Lap_G} = \spmtx{ e^{-t\lgr_1}\\
	&e^{-t\lgr_2}\\
	&&e^{-t\lgr_3}\\
	&&&\ddots}}
with respect to the basis consisting of eigenfunctions of the laplacian. 
\end{para}

\begin{para}[The Heat Kernel]\label{par:heatknlongp}
The heat diffusion operator is, in fact, an integral operator on $L^2(G)$ with a \cinfty-kernel. That means, there is some $K_t\in\cinfty(M\times M)$, such that 
\eq{ (e^{t\Lap_G}f)(x) = \int_GK_t(x,y)f(y)\,\vol(y) \label{eq:heatkernel}}
for any $f\in L^2(G)$. The kernel $K_t$ is called the \de{heat kernel} of $\Lap_G$. Owing to the equivariance of  $\Lap_G$, we have $K_t(x,y)=K_t(e,x\inv y)$, where $e$ is the identity of $G$ any $x,y$ are arbitrary points in $G$. So the heat kernel is completely determined by the function
\dm{ k_t(x) \ldef K_t(e,x). }
We will call this the \de{heat convolution kernel} for $\Lap_G$. With it, equation \eqref{eq:heatkernel} can be rephrased as
\eq{ (e^{t\Lap_G}f)(x) = \int_Gk_t(x\inv y)f(y)\,\vol(y).}

The trace of the heat diffusion operator, $Z(t)=\tr(e^{t\Lap_G})$, is called the \de{partition function} or the \de{heat-trace} of $\Lap_G$. It can be calculated in terms of the heat kernel as follows:
\eq{ Z(t) = \int_GK_t(x,x)\,\vol(x)=k_t(e)\vol(G).\label{eq:heattrace}}

The convolution kernel $k_t$ admits an \de{asymptotic expansion}
\eq{ k_t \sim h_t (a_0+a_1t+a_2t^2+\dotsb) \label{eq:asympexp}}
as $t\rar0+$, where $h_t$ is the gaussian kernel and $a_i$  are smooth functions on $M$. The gaussian kernel $h_t$, under the exponential chart near $e\in G$, takes the form
\dm{ h_t(X)=\frac{e^{-\|X\|^2/4t}}{(4\pi t)^{\dim M/2}}. }
The asymptotic expansion \eqref{eq:asympexp} means that, for each nonnegative integers $r$, $N$, and $n$, there is a constant $C$ such that
\dm{ \Bigl\|k_t - h_t\sum^{N}_{i=0}a_it^i\Bigr\|_{C^r} \le C |t|^n }
for sufficiently small $t$. Here $\|\cdot\|_{C^r}$ denotes the usual norm on  $C^r(M)$. The asymptotic series $s_t\ldef h_t\sum^{\infty}_{i=0}a_it^i$ is a formal solution to the differential equation
\eq{ (\prt_t+\Lap_G)s_t = 0 \label{eq:diffeqasympexp}}
under the condition $s_t(e)=1$. This gives a family of differential equations that can be solved inductively:
\dm{ (\prt_t+\Lap_G)h_t\sum_{i=0}^{k}a_it^i =h_t t^k\Lap_Ga_k. }
\end{para}

\subsection{Algebraic Aspects}

\begin{para}[The Universal Enveloping Algebra]
The universal enveloping algebra $\uen(\grg)$ of $\grg$ is constructed by first taking the tensor algebra $T(\grg)$ of $\grg$ and then taking the quotient by the ideal $I(\grg)$ generated by the elements of the form $X\tsr Y-Y\tsr X - [X,Y]$:
\dm{ \uen(\grg) = T(\grg)/I(\grg).}
The adjoint action of $X\in\grg$ on $\grg$ extends to $\uen(\grg)$ as a derivation. The invariant subspace of $\uen(\grg)$ under all such inner derivations is the center of the universal enveloping algebra;
\dm{ \cen(\grg)=\uen(\grg)^\grg.}

Suppose we have a simple tensor $X_1\dotsm X_n$ in $\uen(\grg)$. It generates the left-invariant differential operator $\tilde X_1\dotsm \tilde X_n$, where $\tilde X_i$ denotes the left-invariant vector field on $G$ generated by $X_i\in\grg$. This gives an algebra isomorphism between $\uen(\grg)$ and the space $D(G)$ of left-invariant differential operators. Under this bijection, the  center $\cen(\grg)$ of the universal enveloping algebra corresponds to the subalgebra of bi-invariant differential operators on $G$.
\end{para}

\begin{para}[The Casimir Element]
We pointed out earlier that the laplacian $\Lap_G$ is a bi-invariant operator. If $\set{X_i}_{i=1}^{n}$ ($n=\dim\grg$) is an \onb. for $\grg$, then we claim that
\eq{ \Lap_G =\sum_{i=1}^{n}\tilde X_i\tilde X_i. \label{eq:lapiscas}}
In other words, the element in $\uen(\grg)$ that corresponds to $\Lap_G$ is the \de{Casimir element}:
\dm{ \Cas = \sum_{i=1}^{n} X_iX_i.}
Since the differential operators on both sides of \eqref{eq:lapiscas} are left-invaraint, it is enough to check their equality at $e\in G$. To that end, take the exponential coordinate system $(x_1,\dotsc,x_n)$ centered at $e$. In other words, the coordinates $(x_1,\dotsc,x_n)$ correspond to the point $\exp\bigl(\sum_{i=1}^n x_iX_i\bigr)$ in $G$. In this coordinate system, we have
\eq{ \biggl.\sum_{i=1}^{n}\tilde X_i\tilde X_i f\,\biggr|_e = \biggl.\sum_{i=1}^{n}\piff^2f x_i^2.\biggr|_0.}
We need to show that the \rhs. is equal to $\Lap_Gf$ at the identity; in other words, we need to verify that the expression of $\Lap_G$ at the identity under the exponential chart is $\sum_{i=1}^n\left.\frac{\prt}{\prt x_i}\right|_0$. To see that this is the case, recall the riemannian exponential map $\Exp:\grg\rar G$ arising from the metric (see \cite{helgason}*{Ch.1, $\S$6} for details). This is, in general, different from the Lie-theoretic exponential map $\exp$ which has nothing to do with the metric. But the two exponential maps do agree if  the metric is bi-invariant, which can be seen as follows. Let $\del$ be the riemannian connection so that it satisfies $\del_{\wt X} \brk{\wt Y,\wt Z}=\brk{\del_{\wt X}\wt Y,\wt Z}+\brk{\wt Y,\del_{\wt X}\wt Z}$. Using the identity $2\brk{\del_{\wt X}{\wt Y},{\wt Z}} ={\wt X}\brk{{\wt Y},{\wt Z}}+{\wt Y}\brk{{\wt Z},{\wt X}}-{\wt Z}\brk{{\wt X},{\wt Y}}+\brk{[{\wt X},{\wt Y}],{\wt Z}}-\brk{[{\wt Y},{\wt Z}],{\wt X}}+\brk{[{\wt Z},{\wt X}],{\wt Y}}$ and the skew-symmetricity of the $\ad(\grg)$-action, one can check that 
\eq{\del_{\wt X}{\wt Y}=\hlf[{\wt X},{\wt Y}] \label{eq:biinvlevicicon} }
holds for all $X,Y$ in $\grg$. In particular, $\del_{\wt X}{\wt X}=0$. It follows \cite{helgason}*{Ch.2, Prop.1.4} that the geodesic $\cgr_X(t)$, such that $\cgr_X(0)=e$ and $\cgr_X'(0)=X$, is a group homomorphism $\R\rar G$. By the uniqueness of $1$-parameter subgroups, we have $\cgr_X(t)=\exp(tX)$. This implies that the riemannian exponential map is identical to the Lie-theoretic exponential map. As a consequence, the matrix $[g_{ij}]$ of the metric under the exponential chart satisfies $g_{ij}(e)=\dgr_{ij}$ (Kronecker delta) and $\prt_kg_{ij}(e)=0$. Therefore, by equation \eqref{eq:lapbeltlocex}, we have
\eq{ \biggl.\Lap_Gf\,\biggr|_e = \biggl.\sum_{i=1}^{n}\piff^2f x_i^2.\biggr|_0.}
Hence $\Lap_G$ agrees with $\sum_{i=1}^n\wt X_i\wt X_i$ at the identity, and this proves that these two invariant operators are equal everywhere on $G$.
\end{para}

\begin{para}[The Duflo Isomorphism]
Let $S(\grg)$ be the symmetric algebra of $\grg$. The adjoint action of $X\in\grg$ on $\grg$ extends, as an inner derivation, to $S(\grg)$. Denote by $S(\grg)^\grg$ the subalgebra of $S(\grg)$ that is invariant under all such inner derivations. We identify $S(\grg)^\grg$ with the constant coefficient differential operators on $\grg$. Duflo \cite{duflo} showed that there is an algebra isomorphism
\dm{ \Duf:S(\grg)^\grg\rar\cen(\grg). }
If we view, for the moment, $\uen(\grg)$ as the convolution algebra of distributions on $G$ supported at $e\in G$ and $S(\grg)$ as the convolution algebra of distributions on $\grg$ supported at $0\in\grg$, then $\Duf$ is $j\cdot \exp_*$, that is, the push-forward along the exponential map followed by the multiplication by the function $j(X) = \det^{1/2}(\frac{\sinh \ad_X/2}{\ad_X/2})$. We note that the push-forward map $\exp_*$ alone gives the \poincare.-Birkhoff-Witt isomorphism $S(\grg)^\grg\rar \cen(\grg)$, which is only a vector space isomorphism.
\end{para}

\section{The Asymptotic Expansion of the Heat Kernel of the Laplacian on a Compact Lie Group}

\begin{para}
Let $\Lap_\grg$ be the laplacian of the euclidean space $\grg$; if $\set{X_i}_{i=1}^n$ is an \onb. for $\grg$, then $\Lap_\grg=\sum_{i=1}^nX_iX_i$. This is an element of $S(\grg)^\grg$, which we identify as the space of constant coefficient differential operators on $\grg$. The image of $\Lap_\grg$ under the Duflo isomorphism is
\eq{ \Duf(\Lap_\grg)= \Cas +\frac{1}{24}\tr_\grg(\Cas), \label{eq:lapduf}}
where $\tr_\grg$ denotes the trace for the linear operators on $\grg$ obtained by extending the adjoint representation $\ad:\grg\rar\End(\grg)$ to the universal enveloping algebra $\uen(\grg)$. Equation \eqref{eq:lapduf} can be proved in more than one way. For a Lie-algebraic proof, we refer to the work of Alekseev and Meinrenken \cite{alekmein}. Recall that $\Cas$, under the identification of $\cen(\grg)$ with the space of bi-invariant differential operators on $G$, corresponds to the laplacian $\Lap_G$ on $G$. And, by a result of Kostant \cite{kostant}*{Eq.1.85}, we have 
\eq{ \frac{1}{24}\tr_\grg\Cas = -\brk{\rho,\rho}\label{eq:kostantcastr}}
where $\ecbrk$  is the inner product---induced from the metric---on the dual space $\grt^*$ of a maximal abelian subalgebra $\grt$ of $\grg$, and $\rho\in\grt^*$ is the half the sum of the positive roots of $G$. Hence, we have
\eq{ \Duf(\Lap_\grg)= \Lap_G-\brk{\rho,\rho}. \label{eq:duflapcptg}}
\end{para}

\begin{lem}\label{lem:duflapcptg}
Let $\Duf(\Lap_\grg)^{\exp}$ be the differential operator defined near a \nbhd. of $0\in\grg$ by expressing the differential operator $\Duf(\Lap_\grg)$ on $G$ under the exponential chart near the identity $e\in G$. We have
\dm{ \Duf(\Lap_\grg)^{\exp} = j\inv \comp\Lap_\grg\comp j,}
where $j$ and $j\inv$ above indicates the multiplication by the function $j(X) = \det^{1/2}(\frac{\sinh \ad_X/2}{\ad_X/2})$ and its reciprocal, respectively.
\end{lem}
\begin{proof}
Recall that the Duflo isomorphism is given by $j\cdot\exp_*$ under the identification of $\uen(\grg)$ with the convolution algebra of distributions on $G$ supported at $e\in G$ and $S(\grg)$ with the convolution algebra of distributions on $\grg$ supported at $0\in\grg$. Thus, back in the language of differential operators, $\left.\Duf(\Lap_\grg)f\right|_{e} = \left.\Lap_{\grg}(jf^{\exp})\right|_0$ where $f^{\exp}\ldef f\comp\exp$. Since $j(0)=1$, we have $\left.\Duf(\Lap_\grg)^{\exp}f^{\exp}\right|_{0} = \left.(j\inv\comp\Lap_{\grg}\comp j)f^{\exp})\right|_0$. Thus, $\Duf(\Lap_\grg)^{\exp}$ agrees with $j\inv\comp \Lap_\grg\comp j$ at $0\in\grg$. This implies that they agree on a \nbhd. of $0\in\grg$, provided that the operators $\Duf(\Lap_\grg)^{\exp}$ and $j\inv\comp\Lap_\grg\comp j$ both define invariant differential operators near $e\in G$. The invariance of $\Duf(\Lap_\grg)^{\exp}$ is clear from equation \eqref{eq:duflapcptg}. For the invariance of $j\inv\comp \Lap_\grg\comp j$, we refer to \cite{helgasongg}*{Ch.II, Eq.71, p.273}.
%%% $\Lap_G^{\exp} = j\inv \comp\Lap_\grg\comp j+ \brk{\rho,\rho}$. (See \cite{helgasongg}*{Ch.II, Eq.71, p.273}.) Hence, $\Duf(\Lap_\grg)^{\exp} = j\inv \comp\Lap_\grg\comp j$.
\end{proof}

\begin{lem}\label{thm:asymhtkduflap}
Let $p_t$ be the convolution kernel of $e^{t\Duf(\Lap_\grg)}$. Then $p_t^{\exp}\ldef p_t\comp \exp$ has the asymptotic expansion
\dm{ p_t^{\exp}\sim h_t j\inv,\qquad t\rar0+, }
valid in some \nbhd. of $0\in\grg$, where $h_t$ is the gaussian kernel on $\grg$.
\end{lem}
\begin{proof}
Let $s_t\ldef h_t\sum_{i=0}^{\infty}a_it^i$ be the asymptotic expansion for $p_t^{\exp}$. It is the formal solution to 
\eq{ (\prt_t+\Duf(\Lap_\grg)^{\exp})s_t =0, \label{eq:liegphteq}}
where $\Duf(\Lap_\grg)^{\exp}$ is the differential operator $\Duf(\Lap_\grg)$ expressed in the exponential chart near the identity $e\in G$. By Lemma \thref{lem:duflapcptg}, the differential equation \eqref{eq:liegphteq} is equivalent to
\dm{ (\prt_t-j\inv\comp \Lap_\grg\comp j)s_t=0.}
We need to show that $h_t/j$ satisfies this differential equation. This is easily done by invoking the fact that $h_t$ satisfies the heat equation $(\prt_t-\Lap_\grg)h_t=0$; indeed,
\dm{ (\prt_t-j\inv\comp \Lap_\grg\comp j)h_t/j = j\inv \prt_th_t - j\inv \Lap_\grg h_t=j\inv(\prt_t-\Lap_\grg)h_t=0.\qedhere}
\end{proof}

\begin{lem}\label{lem:scacurvcptliegp}
Let $G$ be a compact connected Lie group equipped with a bi-invariant metric. The scalar curvature $S$ of $G$ is equal to $-\frac{1}{4}\tr_\grg(\Cas)$.
\end{lem}
\begin{proof}
We pointed out in \eqref{eq:biinvlevicicon} that the riemannian connection $\del$ on $G$ satisfies $\del_{\wt X}\wt Y = \hlf[\wt X,\wt Y]$. A routine calculation shows that the Riemann curvature tensor $\mathrm{Rm}$ satisfies 
\dm{ \mathrm{Rm}(\wt X,\wt Y,\wt Z,\wt W) = -\frac{1}{4}\brk{[[X,Y],Z],W} =-\frac{1}{4}\brk{[X,Y],[Z,W]},}
where, for the last equality, we have used the fact that $\ad(\grg)$-action is skew-symmetric. Let $\set{X_i}_{i=1}^{\dim\grg}$ be an \onb. for $\grg$. For the scalar curvature $S$, we have
\al{ S &=\sum_{i,j=1}^{\dim\grg} \mathrm{Rm}(\wt X_i,\wt X_j,\wt X_j,\wt X_i)=-\frac{1}{4}\sum_{i,j=1}^{\dim\grg}\brk{[[X_i,X_j],X_j],X_i}\\
	&=-\frac{1}{4}\sum_{i,j=1}^{\dim\grg}\brk{\ad(\Cas)X_i,X_i}=-\frac{1}{4}\tr_\grg(\Cas).\qedhere}
\end{proof}

\begin{thm}\label{thm:asymhtklapcptg}
Let $G$ be a compact connected Lie group equipped with a bi-invariant metric. Let $k_t$ be the heat convolution kernel for the laplacian on $G$. Then $k_t^{\exp} = k_t\comp\exp$ has the asymptotic expansion
\dm{ k_t^{\exp}\sim  \frac{h_t}{j}e^{tS/6} }
as $t\rar0+$, valid in a \nbhd. of $0\in\grg$, where $S$ is the scalar curvature, $j$ is the function defined by the power series $j(X)=\det^{1/2}(\frac{\sinh \ad_X/2}{\ad_X/2})$ and $h_t$ is the gaussian kernel on $\grg$, that is, $h_t(X) = e^{-\|X\|/4t}/(4\pi t)^{\dim\grg/2}$. 
\end{thm}
\begin{proof}
Owing to equations \eqref{eq:kostantcastr}, \eqref{eq:duflapcptg}, and Lemma \thref{lem:scacurvcptliegp}, we have $e^{t\Lap_G}=e^{tS/6}e^{t\Duf(\Lap_\grg)}$. This implies $k_t=e^{tS/6}p_t$ where $p_t$ is the convolution kernel for $e^{t\Duf(\Lap_\grg)}$. The theorem now follows from Lemma \thref{thm:asymhtkduflap}.
\end{proof}

\begin{cor}
Let $G$ be a compact connected Lie group equipped with a bi-invariant metric. Let $S$ be the scalar curvature. The heat-trace $Z(t)=\tr(e^{t\Lap_G})$ of the laplacian $\Lap_G$ on $G$ has the asymptotic expansion
\dm{ Z(t) \sim \vol(G)\,e^{tS/6}}
as $t\rar0+$.
\end{cor}
\begin{proof}
This follows from Theorem \thref{thm:asymhtklapcptg} and equation \eqref{eq:heattrace}.
\end{proof}

\end{document}